\documentclass[12pt,reqno]{amsart}

\usepackage{ifthen}
\usepackage{amsmath}
\usepackage{amssymb}

\usepackage{verbatim}

\newcommand{\cB}{{\mathcal B}}

\newcommand{\sA}{{\mathsf A}}
\newcommand{\sB}{{\mathsf B}}
\newcommand{\sD}{{\mathsf D}}
\newcommand{\sI}{{\mathsf I}}
\newcommand{\sL}{{\mathsf L}}
\newcommand{\sM}{{\mathsf M}}
\newcommand{\sU}{{\mathsf U}}
\newcommand{\sV}{{\mathsf V}}

\newcommand{\Lam}{\Lambda}

\newcommand{\F}{{\mathbb F}}
\newcommand{\Z}{{\mathbb Z}}

\newcommand{\eps}{\varepsilon}
\renewcommand{\phi}{\varphi}
\newcommand{\sig}{\sigma}

\newcommand{\lpr}{\left(}
\newcommand{\rpr}{\right)}
\newcommand{\lfl}{\left\lfloor}
\newcommand{\rfl}{\right\rfloor}

\newcommand{\stm}{\setminus}
\newcommand{\seq}{\subseteq}

\newcommand{\longc}{,\ldots,}
\newcommand{\longe}{=\dotsb=}
\newcommand{\longp}{+\dotsb+}
\newcommand{\longm}{-\dotsb-}
\newcommand{\longu}{\cup\dotsb\cup}
\newcommand{\longi}{\cap\dotsb\cap}

\DeclareMathOperator{\rk}{rk}
\DeclareMathOperator{\diag}{diag}

\theoremstyle{plain}

\newtheorem{lemma}{Lemma}
\newtheorem{theorem}{Theorem}
\newtheorem{proposition}{Proposition}

\theoremstyle{remark}
\newtheorem*{remark}{Remark}

\newcommand{\reft}[1]{~\ref{t:#1}}
\newcommand{\refp}[1]{~\ref{p:#1}}

\newcommand{\refs}[1]{~\ref{s:#1}}
\newcommand{\refb}[1]{~\cite{b:#1}}

\title{Additive Bases in Abelian Groups}

\author{Vsevolod F. Lev}
\email{seva@math.haifa.ac.il}
\address{Department of Mathematics, The University of Haifa at Oranim,
    Tivon 36006, Israel}

\author{Mikhail E. Muzychuk}
\email{muzy@netanya.ac.il}
\address{Department of Computer Science and Mathematics,
    Netanya Academic College, Netanya 42365, Israel}

\author{Rom Pinchasi}
\email{room@math.technion.ac.il}
 \address{Department of Mathematics, Technion --- Israel Institute of
    Technology, Haifa 32000, Israel}

\begin{document}
\baselineskip=16pt

\begin{abstract}
Let $G$ be a finite, non-trivial abelian group of exponent $m$, and
suppose that $B_1\longc B_k$ are generating subsets of $G$. We prove that
if $k>2m\ln\log_2 |G|$, then the multiset union $B_1\longu B_k$ forms an
additive basis of $G$; that is, for every $g\in G$ there exist
 $A_1\seq B_1\longc A_k\seq B_k$ such that $g=\sum_{i=1}^k\sum_{a\in A_i} a$.
This generalizes a result of Alon, Linial, and Meshulam on the additive
bases conjecture.

As another step towards proving the conjecture, in the case where
 $B_1\longc B_k$ are finite subsets of a vector space we obtain lower-bound
estimates for the number of distinct values, attained by the sums of the
form $\sum_{i=1}^k \sum_{a\in A_i} a$, where $A_i$ vary over all subsets
of $B_i$ for each $i=1\longc k$.

Finally, we establish a surprising relation between the additive bases
conjecture and the problem of covering the vertices of a unit cube by
translates of a lattice, and present a reformulation of (the strong form
of) the conjecture in terms of coverings.
\end{abstract}

\maketitle

\section{Introduction and statement of the results}

Given an abelian group $G$ and a multiset $B=\{b_1\longc b_n\}\seq G$, we say
that $B$ is an \emph{additive basis} of $G$ if for every group element
 $g\in G$ there exists an index set $I\seq[1,n]$ such that
$g=\sum_{i\in I} b_i$.

Jaeger et al conjectured in \refb{jlpt} that for any prime $p$ there is an
integer $k(p)>0$ with the property that if $B_1\longc B_{k(p)}$ are (linear)
bases of the finite-dimensional vector space $V$ over the $p$-element field,
then the multiset union $B_1\longu B_{k(p)}$ is an additive basis of $V$.
This statement is known as the \emph{additive basis conjecture}. Alon et al
proved in \refb{alm}, in two different ways, that there exists a function
$c(p)$ such that if $|V|=p^n$, $k>c(p)\ln n$, and $B_1\longc B_k$ are bases
of $V$, then the multiset union $B_1\longu B_k$ is an additive basis of $V$.
One of their proofs, using the polynomial method and properties of the
permanent, shows that $B_1\longu B_k$ is an additive basis of $V$, provided
that $k\ge(p-1)\ln n+p-2$; another proof, using exponential sums, requires
$k\ge(p^2/2)\ln 2pn+1$. In this paper we introduce yet another approach
allowing us to establish the following generalization.

\begin{theorem}\label{t:main}
Let $G$ be a finite, non-trivial abelian group of exponent $m$, and suppose
that $B_1\longc B_k$ are generating subsets of $G$. If $k>2m\ln\log_2 |G|$,
then the multiset union $B_1\longu B_k$ is an additive basis of $G$.
\end{theorem}

Our argument is based on the following interpretation of the problem. For a
subset $B$ of an abelian group $G$ write
  $$ \textstyle B^\ast := \Big\{ \sum_{a\in A} a \colon A\seq B,
                                                       \,|A|<\infty \Big\} $$
(the \emph{subset sum set} of $B$), and given subsets $B_1\longc B_k\seq G$
let
  $$ B_1\longp B_k
                  := \{ b_1\longp b_k\colon b_1\in B_1\longc b_k\in B_k \} $$
(the \emph{sumset} of $B_1\longc B_k$). Clearly, in order for the multiset
union $B_1\longu B_k$ to form an additive basis of $G$, it is necessary and
sufficient that $B_1^\ast\longp B_k^\ast=G$. Accordingly, Theorem \reft{main}
can be equivalently restated as follows. \addtocounter{theorem}{-1}
\renewcommand{\thetheorem}{\arabic{theorem}$'$}
\begin{theorem}
Let $G$ be a finite, non-trivial abelian group of exponent $m$, and suppose
that $B_1\longc B_k$ are generating subsets of $G$. If $k>2m\ln\log_2 |G|$,
then $B_1^\ast\longp B_k^\ast=G$.
\end{theorem}
\renewcommand{\thetheorem}{\arabic{theorem}}

We notice that if $B_1\longc B_k$ are bases of a vector space, then the
sumset $B_1^\ast\longp B_k^\ast$ has a transparent geometric meaning: namely,
it is the  Minkowski sum of the vertex sets of the parallelepipeds, spanned
by $B_1\longc B_k$. One may hope that studying sumsets of this form can
eventually lead to a proof of the additive basis conjecture. We establish two
results in this direction.

\begin{theorem}\label{t:char0}
Let $k$ be a positive integer. If $B_1\longc B_k$ are finite, non-empty
subsets of the vector space $V$ over a field of infinite characteristic, then
  $$ |B_1^\ast\longp B_k^\ast| \ge 2^{\rk(B_1)}\,
                             (3/2)^{\rk(B_2)} \dotsb ((k+1)/k)^{\rk(B_k)}. $$
In particular, if $\dim V=n$ and $B_1\longc B_k$ are bases of $V$, then
  $$ |B_1^\ast\longp B_k^\ast| \ge (k+1)^n. $$
\end{theorem}

Notice, that in the second estimate of the theorem equality is attained if
$B_1\longe B_k$.

\begin{theorem}\label{t:charp}
If $B_1$ and $B_2$ are finite subsets of a vector space $V$ over a field of
infinite or odd characteristic, then
  $$ |B_1^\ast+B_2^\ast| \ge \lpr \frac83\rpr^{(\rk(B_1)+\rk(B_2))/2}. $$
In particular, if $\dim V=n$ and $B_1,B_2$ are bases of $V$, then
  $$ |B_1^\ast+B_2^\ast| \ge \lpr \frac83\rpr^n. $$
\end{theorem}

It is difficult to expect that the constant $8/3$ is best possible in this
context. On the other hand, it cannot be replaced by a value larger than
$\sqrt 8$, at least for the underlying field of characteristic $3$: for, if
$V$ is an even-dimension vector space over such a field, and if
 $B_1=\{e_1\longc e_{2l}\}$ is a basis of $V$, then for the basis
  $$ B_2 := \{e_1+e_2,e_1-e_2, \, e_3+e_4,e_3-e_4
                             \longc e_{2l-1}+e_{2l},e_{2l-1}-e_{2l}\} $$
we have $|B_1^*+B_2^*|=8^l$.

Yet another, completely different approach is presented in Section
\refs{lattice}, where the additive bases conjecture is interpreted in terms
of coverings of the vertices of the unit cube by translates of an integer
lattice. We postpone the discussion and exact statement of the result to
avoid aggregating notation and terminology at this stage.

\section{The proofs}

\begin{proof}[Proof of Theorem \reft{main}$'$]
For $m=2$ the assertion is immediate, as in this case $B^\ast=G$ for any
generating subset $B\seq G$. Assume for the rest of the proof that
 $m\ge 3$, and for $j\in[1,k]$ let $S_j:=B_1^\ast\longp B_j^\ast$.

The key ingredient of our argument is the Ruzsa sum triangle inequality,
which says that for any positive integer $n$ and any finite subsets
$A_0,A_1\longc A_n$ of an abelian group one has
  $$ |A_0|^{n-1}|A_1\longp A_n| \le |A_0+A_1|\dotsb |A_0+A_n|; $$
see \refb{r}. Applying this inequality with $n=m-1,\,A_0=S_{j-1}$, and
$A_1\longe A_n=B_j^\ast$, we obtain
  $$ |S_{j-1}|^{m-2}|(m-1)B_j^\ast|
                         \le |S_{j-1}+B_j^\ast|^{m-1};\quad j\in[2,k]. $$
As $(m-1)B_j^\ast=G$ and $S_{j-1}+B_j^\ast=S_j$, we derive that
   $$ |S_j| \ge |S_{j-1}|^{1-\frac1{m-1}} |G|^\frac1{m-1} $$
and consequently
  $$ \frac{|G|}{|S_j|} \le \lpr \frac{|G|}{|S_{j-1}|} \rpr^{1-\frac1{m-1}} $$
for every $j\in[2,k]$. Iterating, we obtain
  $$ \ln \frac{|G|}{|S_j|}
          \le \lpr 1-\frac1{m-1}\rpr^{j-1} \ln\frac{|G|}{|S_1|}
                \le e^{-\frac{j-1}m} \ln\frac{|G|}{|S_1|}; \quad j\in[1,k]. $$
Let $r$ denote the rank of $G$. Since $|G|\le m^r$ and $|S_1|=|B_1^\ast|\ge
2^r$ (for if $C$ is a minimal generating subset of $B_1$, then
$|B_1^\ast|\ge|C^\ast|=2^{|C|}\ge 2^r$), using some basic calculus it is not
difficult to deduce that
  $$ \ln \frac{|G|}{|S_j|} < e^{-\frac{j+1}m} \ln |G|; \quad j\in[1,k]. $$
Thus, if $j=\lfl m\ln\log_2|G|\rfl$ (so that $k>2j>0$), then
  $$ |B_1^\ast\longp B_j^*|=|S_j|>\frac12\,|G|, $$
and similarly
  $$ |B_{j+1}^\ast\longp B_{2j}^*|>\frac12\,|G|. $$

We now use the fact that if two sets $S,T\seq G$ satisfy $|S|+|T|>|G|$, then
for any element $g\in G$ the sets $S$ and $g-T$ have non-empty intersection
in view of
  $$ |S\cap(g-T)| = |S|+|g-T|-|S\cup(g-T)| \ge |S|+|T|-|G| > 0; $$
hence $g\in S+T$ and therefore $S+T=G$. Applying this to the sets
$S:=B_1^\ast\longp B_j^*$ and $T:=B_{j+1}^\ast\longp B_{2j}^*$, we get
  $$ B_1^\ast\longp B_{2j}^* = G, $$
which implies the result.
\end{proof}

\begin{proof}[Proof of Theorem \reft{char0}]
We use induction on $k$, and for any fixed value of $k$ induction on
  $$ \min\{\rk(B_1)\longc\rk(B_k)\}. $$
The case $k=1$ follows from the observation that if $r=\rk(B_1)$ and
$b_1\longc b_r\in B_1$ are linearly independent, then all $2^r$ sums
  $$ \textstyle \sum_{i\in I} b_i;\quad I\seq[1,r] $$
are pairwise distinct. The case where $k\ge 2$ and
$\min\{\rk(B_1)\longc\rk(B_k)\}=0$ follows easily by the induction
hypothesis. Suppose, therefore, that $k\ge 2$ and $\rk(B_1)\longc\rk(B_k)$
are all positive. Fix arbitrarily $b\in B_k$ and write $B_0:=B_k\stm\{b\}$.
Choose a linear subspace $L$ complementing ${\rm Sp}\,\{b\}$ to the whole
vector space, and let $\pi$ denote the projection onto $L$ along $b$. We have
then
  $$ B_1^\ast\longp B_k^\ast
          = (B_1^\ast\longp B_{k-1}^\ast+B_0^\ast)
                         \cup(B_1^\ast \longp B_{k-1}^\ast+B_0^\ast+b), $$
and since
\begin{multline*}
  |(B_1^\ast \longp B_{k-1}^\ast+B_0^\ast+b)
                          \stm (B_1^\ast\longp B_{k-1}^\ast+B_0^\ast)| \\
      \ge |\pi(B_1^\ast \longp B_{k-1}^\ast+B_0^\ast)|
                                = |(\pi(B_1))^\ast \longp (\pi(B_k))^\ast|,
\end{multline*}
it follows by the induction hypothesis (and in view of
$\rk(B_0)\ge\rk(B_k)-1$ and $\rk(\pi(B))\ge\rk(B)-1$, for any vector set $B$)
that
\begin{align*}
  |B_1^\ast\longp B_k^\ast|
    &= |B_1^\ast\longp B_{k-1}^\ast+B_0^\ast| \\
    &\hspace{.5in} + |(B_1^\ast \longp B_{k-1}^\ast+B_0^\ast+b)
                          \stm (B_1^\ast\longp B_{k-1}^\ast+B_0^\ast)| \\
    &\ge \prod_{j=1}^{k-1}\lpr 1+\frac1j\rpr^{\rk(B_j)}
                     \cdot \lpr 1+\frac1k \rpr^{\rk(B_k)-1}
                       +\ \prod_{j=1}^{k}\lpr 1+\frac1j\rpr^{\rk(B_j)-1} \\
    &= \prod_{j=1}^{k}\lpr 1+\frac1j\rpr^{\rk(B_j)},
\end{align*}
as desired.
\end{proof}

For a finite subset $B$ of an abelian group, let
  $$ T(B) := \{(b_1,b_2,b_3,b_4) \in B\times B\times B\times B
                                               \colon b_1+b_2=b_3+b_4\}. $$
The proof of Theorem \reft{charp} is based on
\begin{proposition}\label{p:sumset} For any finite, non-empty subsets $A$ and
$B$ of an abelian group we have
  $$ |A+B| \ge \frac{|A|^2|B|^2}{\sqrt{T(A)T(B)}}. $$
\end{proposition}

\begin{proof}
For a group element $z$ write
\begin{align*}
  \nu_{A-A}(z) &:= |\{(a_1,a_2)\in A\times A\colon z=a_1-a_2\}|, \\
  \nu_{B-B}(z) &:= |\{(b_1,b_2)\in B\times B\colon z=b_1-b_2\}|, \\
\intertext{and}
  \nu_{A+B}(z) &:= |\{ (a,b)\in A\times B\colon z=a+b \}|.
\end{align*}
By the Cauchy-Schwarz inequality, we have
  $$ \frac{|A|^2|B|^2}{|A+B|}
         = \frac1{|A+B|}\,\lpr \sum_z \nu_{A+B}(z) \rpr^2
                                   \le \sum_z \big( \nu_{A+B}(z)\big)^2. $$
The sum in the right-hand side is the number of solutions of the equation
$x_1+y_1=x_2+y_2$ in the variables $x_1,x_2\in A$ and $y_1,y_2\in B$.
Rewriting this equation as $x_1-x_2=y_2-y_1$ we see that the number of its
solutions is equal to
\begin{align*}
  \sum_z \nu_{A-A}(z)\nu_{B-B}(z)
         &\le \lpr \sum_z \big( \nu_{A-A}(z) \big)^2 \rpr^{1/2}
                     \lpr \sum_z \big( \nu_{B-B}(z) \big)^2 \rpr^{1/2} \\
         &=   \sqrt{T(A)T(B)},
\end{align*}
and the result follows.
\end{proof}

\begin{proof}[Proof of Theorem \reft{charp}]
Removing dependent vectors from the sets $B_i$, we assume without loss of
generality that $\rk(B_i)=|B_i|$ and $|B_i^\ast|=2^{|B_i|}$ for
$i\in\{1,2\}$. We apply Proposition \refp{sumset} to the sets $B_1^\ast$ and
$B_2^\ast$. The quantities $T(B_i^\ast)$ can be explicitly calculated as
follows. For a group element $z$ let
  $$ \nu_i(z) := | \{ (b',b'')\in B_i^\ast\times B_i^\ast
                               \colon z=b'-b'' \} |; \quad i\in\{1,2\}, $$
so that
  $$ T(B_i^\ast)=\sum_{z\in B_i^\ast-B_i^\ast} (\nu_i(z))^2. $$
Fix $i\in\{1,2\}$, set $n:=|B_i|$, and write $B_i=\{b_1\longc b_n\}$. The
condition $z\in B_i^\ast-B_i^\ast$ means that $z=\eps_1b_1\longp\eps_nb_n$
with $\eps_1\longc\eps_n\in\{-1,0,1\}$, and it is easily seen that in this
case
  $$ \nu_i(z) = 2^{n-|\eps_1|\longm|\eps_n|}. $$
Consequently, we have
\begin{align*}
  T(B_i^\ast)
    &= \sum_{\eps_1\longc\eps_n=-1}^1 2^{2n-2|\eps_1|\longm2|\eps_n|} \\
    &= 2^{2n} \prod_{i=1}^n \sum_{\eps_i=-1}^1 2^{-2|\eps_i|} \\
    &= 6^n
\end{align*}
and it follows that
  $$ |B_1^\ast+B_2^\ast|
         \ge \frac{2^{2|B_1|+2|B_2|}}{\sqrt{6^{|B_1|+|B_2|}}}
                 = \lpr\frac83\rpr^{(|B_1|+|B_2|)/2}
                   = \lpr\frac83\rpr^{(\rk(B_1)+\rk(B_2))/2}. $$
\end{proof}

We remark that the approach used in the proof of Proposition \refp{sumset}
can be combined with character sum technique in the spirit of \refb{alm}, to
show that for any system of bases $B_1\longc B_k$ of an $r$-dimensional
vector space over a field of finite characteristic $p>2$ one has
  $$ |B_1^\ast\longp B_k^\ast| \ge \frac{p^r}{(1+\sig_p(k))^r} \, , $$
where
  $$ \sig_p(k) = \sum_{u=1}^{p-1} \left( \cos\pi\frac up \right)^{2k}. $$
(Hint for the interested reader: if $\nu(g)$ denotes the number of
representations of the vector $g$ in the form $g=b_1\longp b_k$ with
 $b_i\in B_i^\ast$ for $i\in[1,k]$, then the sum $\sum_g(\nu(g))^2$ counts
the number of $2k$-tuples $(x_1\longc x_k,y_1\longc y_k)$ with
 $x_1\longp x_k=y_1\longp y_k$ and $x_i,y_i\in B_i^\ast$ for $i\in[1,k]$.
This number is easily expressed using character sums.) Since
$1+\sig_p(2)=3p/8$ for any prime $p>2$, the estimate above extends the result
of Theorem~\reft{charp}.

\section{Additive bases via lattice coverings}\label{s:lattice}

Suppose that $k\ge 2$ and $r\ge 1$ are fixed integers. Given a prime $p$, we
say that a lattice in $\Z^{kr}$ is \emph{$p$-oblique} if, whenever for some
$i_0\in[1,k]$ a lattice vector
  $$ (z_{11}\longc z_{1r}\longc z_{k1}\longc z_{kr}) $$
satisfies
  $$ z_{ij}\equiv 0\!\!\!\pmod p\ \text{ for all }
                                       i\in[1,k]\stm\{i_0\},\ j\in[1,r], $$
it also satisfies
  $$ z_{i_0j}\equiv 0\!\!\!\pmod p\ \text{ for all } j\in[1,r]. $$
(The term \emph{oblique} is meant as an indication that the lattice is not
aligned with the coordinate hyperplanes.) Furthermore, we define the
\emph{covering number} of a lattice $\Lam\le\Z^{kr}$ to be the minimal number
of translates of $\Lam$, containing in their union all vertices of the unit
cube $[0,1]^{kr}$, and we denote this quantity by $C(\Lam)$; thus,
  $$ C(\Lam) = \min \{ |S|\colon S\seq\Z^{kr},\ \{0,1\}^{kr}\seq S+\Lam \}. $$

We conclude this paper with the following, somewhat surprising, result.
\begin{theorem}\label{t:cover}
For any integer $k\ge 2$ and $r\ge 1$, prime $p$, and field $\F$ of
characteristic $p$, we have
  $$ \min_{B_1\longc B_k\seq\F^r} |B_1^\ast\longp B_k^\ast|
                                         \ge \min_{\Lam\le\Z^{kr}} C(\Lam), $$
where the minimum in the left-hand side extends over all systems of $k$ bases
of $\F^r$, and in the right-hand side over all $p$-oblique lattices in
$\Z^{kr}$. If $k\le|\F|$ then, indeed, equality holds; moreover, in this case
there is a $p$-oblique lattice $p\Z^{kr}\le\Lam\le\Z^{kr}$ with
$\det\Lam=p^r$, on which the minimum in the right-hand side is attained.
\end{theorem}

\begin{remark}
It is not clear to us to what extent the condition $k\le|\F|$ is essential
and whether the assertion of Theorem \reft{cover} fails in a critical way if
this condition is dropped.
\end{remark}

We notice that if $k,r$, and $p$ are as in Theorem \reft{cover}, then the
lattice $\Lam$ of all integer vectors
 $(z_{11}\longc z_{kr})$ with
  $$ z_{1j}\longp z_{kj} \equiv 0\pmod p;\quad j=1\longc r $$
has covering number $C(\Lam)=\min\{(k+1)^r, p^r\}$. To see this, observe that
different translates of $\Lam$ correspond in a natural way to integer vectors
of the form $(z_{11}\longc z_{1r},0\longc 0)$ with
 $z_{11}\longc z_{1r}\in[0,p-1]$, and on the other hand, such a
vector is congruent modulo $\Lam$ to a vector from $\{0,1\}^{kr}$ if and only
if $z_{11}\longc z_{1r}\le k$.

By Theorem \reft{cover}, the additive bases conjecture will follow if one can
show that there are no $p$-oblique lattices with the covering number smaller
than $p^r$; that is, to show that for any prime $p$ there exists an integer
$k=k(p)>0$ such that the set $\{0,1\}^{kr}$ cannot be covered by fewer than
$p^r$ translates of a $p$-oblique lattice, for any integer $r\ge 1$. By all
we know, the conjecture may even be true with $k(p)=p$. Theorem \reft{cover}
shows this strong form of the conjecture is equivalent to the assertion that
one cannot cover the set $\{0,1\}^{pr}$ by fewer than $p^r$ translates of a
$p$-oblique lattice.

\begin{proof}[Proof of Theorem \reft{cover}]
Suppose that $\cB=\{B_1\longc B_k\}$ is a system of bases of $\F^r$. We write
  $$ B_i = \{b_{i1}\longc b_{ir}\};\quad i=1\longc k $$
and define a linear mapping $\phi_\cB\colon \Z^{kr}\to\F^r$ by
  $$ \phi_\cB(x_{11}\longc x_{kr}) = x_{11}b_{11}\longp x_{1r}b_{1r}
                                \longp x_{k1}b_{k1}\longp x_{kr}b_{kr}, $$
so that the sumset $B_1^\ast\longp B_k^\ast$ is the image of $\{0,1\}^{kr}$
under $\phi_\cB$.

Let $\Lam_\cB:=\ker\phi_\cB$; thus, $\Lam_\cB$ is a sublattice of $\Z^{kr}$,
lying above $p\Z^{kr}$, and hence a full-rank sublattice. From the fact that
for any integer $x_{11}\longc x_{k-1,r}$ there are unique $x_{k1}\longc
x_{kr}\in[0,p)$ with $(x_{11}\longc x_{kr})\in\Lam_\cB$, it follows that
$\det\Lam_\cB=p^r$. Furthermore, $\Lam_\cB$ is (evidently) $p$-oblique.

Since for $u,v\in\Z^{kr}$, and in particular for $u,v\in\{0,1\}^{kr}$, the
equality $\phi_\cB(u)=\phi_\cB(v)$ is equivalent to
 $v\equiv u\!\pmod{\Lam_\cB}$, we have
  $$ |B_1^\ast\longp B_k^\ast| = |\phi_\cB(\{0,1\}^{kr})| = C(\Lam_\cB). $$
This proves the first assertion of the theorem, and to complete the proof we
assume that $k\le|\F|$ and show that for every $p$-oblique lattice
$\Lam\le\Z^{kr}$ there exists a system $\cB$ of $k$ bases of $\F^r$ with
$\Lam_\cB\ge\Lam$ and consequently, with $C(\Lam_\cB)\le C(\Lam)$.

Replacing $\Lam$ with $\Lam+p\Z^{kr}$ we can assume that $\Lam$ is of full
rank. Let $\Z^{kr}=V_1\oplus\dotsb\oplus V_k$, where $V_i$ is the rank-$r$
sublattice, spanned over $\Z$ by the ``$i$\/th coordinate block'' (for each
$i=1\longc k$), fix a basis $\{l^{(1)}\longc l^{(kr)}\}$ of $\Lam$, and for
each $t\in[1,kr]$ consider the decomposition
  $$ l^{(t)}=l^{(t)}_1\longp l^{(t)}_k;
                        \quad l^{(t)}_i\in V_i\ \text{ for } i=1\longc k. $$
What we want $\cB$ to satisfy is that $\phi_\cB(l^{(t)})=0$ for every
$t\in[1,kr]$. Writing $\cB=\{B_1\longc B_k\}$ and
 $B_i=\{b_{i1}\longc b_{ir}\}$ for each $i=1\longc k$, the condition to
secure takes the shape
  $$ \sB_1 l^{(t)}_1 \longp \sB_k l^{(t)}_k=0; \quad t\in[1,kr], $$
where $\sB_i$ is the $r\times r$ matrix over $\F$, whose columns are formed
by the vectors $b_{i1}\longc b_{ir}$. Thus, if for each $i=1\longc k$ by
$\sL_i$ we denote the $r\times kr$ matrix over $\Z$, whose columns are formed
by the vectors $l^{(1)}_i\longc l^{(kr)}_i$, then what we ultimately want to
prove is the existence of $\sB_1\longc\sB_k\in{\rm GL}_r(\F)$ satisfying
$\sB_1\sL_1\longp\sB_k\sL_k=0$.

We notice that the matrices $\sL_i$ actually have integer entries, but we
treat them in the last equality, and will continue to treat for the rest of
the proof, as matrices over the $p$-element field $\F_p\le\F$, and we also
seek $\sB_i$ with the entries in $\F_p$. With this convention, the assumption
that $\Lam$ is $p$-oblique (which has not been used yet) guarantees that for
each $i\in[1,k]$, the intersection
  $$ \ker\sL_1\longi\ker\sL_{i-1} \cap\ker\sL_{i+1}\longi\ker\sL_k $$
is contained in $\ker\sL_i$. The proof is concluded by using
\begin{proposition}\label{p:lin}
Suppose that $r,n\ge 1$ and $k\ge 2$ are integers and $\sL_1\longc \sL_k$ are
matrices of size $r\times n$ over a field $\F$ with $|\F|\ge k$. If for each
$i\in[1,k]$ we have
  $$ \bigcap_{j\in[1,k]\colon j\neq i} \ker\sL_j \seq \ker\sL_i, $$
then there exist matrices $\sB_1\longc\sB_k\in{\rm GL}_r(\F)$ satisfying
  $$ \sB_1\sL_1\longp\sB_k\sL_k=0. $$
\end{proposition}
The proof of Proposition \refp{lin} is presented in the Appendix.
\end{proof}

\appendix
\section*{Appendix: Proof of Proposition \refp{lin}.}

The assumptions $\bigcap_{j\in[1,k]\colon j\ne i}\ker\sL_j\seq\ker\sL_i$
shows that if a vector in $\F^n$ is orthogonal to every row of every matrix
$\sL_j$ with $j\in[1,k]\stm\{i\}$, then it is also orthogonal to every row of
the matrix $\sL_i$; in other words, every row of $\sL_i$ is a linear
combination of the rows of $\sL_j$ for $j\in[1,k]\stm\{i\}$. Thus, there
exist matrices $\sA_{ij}$ of size $r\times r$ over the field $\F$, for each
pair of indices $i,j\in[1,k]$ with $i\ne j$, such that
  $$ \sL_i = \sum_{j\in[1,k]\colon j\ne i} \sA_{ij}\sL_j;\quad i\in[1,k]. $$
For $i\in[1,k]$ we define $A_{ii}:=-\sI_r$, where $\sI_r$ is the identity
matrix of size $r\times r$, to get
  $$ \sA_{i1}\sL_1\longp\sA_{ik}\sL_k = 0;\quad i\in[1,k]. $$
(The reader may find it useful to consider the matrices $\sA_{ij}$ being
organized themselves into a $k\times k$ matrix, and the matrices
$\sL_1\longc\sL_k$ written as a ``column vector''.)

If we can find matrices $\sM_1\longc\sM_k$ of size $r\times r$ so that none
of the linear combinations
  $$ \sB_j := \sM_1\sA_{1j}\longp\sM_k\sA_{kj};\quad j\in[1,k] $$
is degenerate, then we are done in view of
  $$ \sum_{j=1}^k \sB_j\sL_j
    = \sum_{i=1}^k \sM_i \sum_{j=1}^k \sA_{ij} \sL_j = 0. $$
We complete the proof of the proposition establishing the existence of such
matrices $\sM_1\longc\sM_k$.

\begin{lemma}\label{l:lemma1}
Suppose that $r,k\ge 1$ are integers and that $\sA_{ij}\ (i,j\in[1,k])$ are
matrices of size $r\times r$ over a field $\F$ such that $A_{ii}$ is
non-degenerate for each $i\in[1,k]$. If $|\F|\ge k$, then there exist
matrices $\sM_1\longc\sM_k$ over $\F$ of size $r\times r$ so that  none of
  $$ \sM_1\sA_{1j}\longp\sM_k\sA_{kj};\quad j\in[1,k] $$
is degenerate.
\end{lemma}

\begin{proof}
We use induction on $k$. The case $k=1$ trivial; suppose that $k\ge 2$.

If $\sA_{k1}\longc\sA_{k\,k-1}$ are all non-degenerate, then we can take
$\sM_1\longe\sM_{k-1}=0$ and $\sM_k=\sI_r$. Assume therefore that one of the
matrices $\sA_{ki}$ with $i\in[1,k-1]$, say $\sA_{k1}$, \emph{is} degenerate.
Applying the induction hypothesis, find $\sM_1\longc\sM_{k-1}$ so that
  $$ \sB_j := \sM_1\sA_{1j}\longp\sM_{k-1}\sA_{k-1\,j};\quad j\in[1,k-1] $$
are non-degenerate, and let
  $$ \sB_k:=\sM_1\sA_{1k}\longp\sM_{k-1}\sA_{k-1\,k}. $$

Recalling that $\sA_{k1}$, and thus also $\sA_{k1}\sB_1^{-1}$, are
degenerate, we find $\sU,\sV\in{\rm GL}_r(\F)$ so that the matrix
$\sV\sA_{k1}\sB_1^{-1}\sU$ is upper-triangular with zeroes on the main
diagonal, and we seek $\sM_k$ in the form $\sM_k=\sU\sD\sV$ with a
diagonal matrix $\sD$. Thus, $\sD$ is to be found so that none of
$\sB_j+\sU\sD\sV\sA_{kj}$ for $j\in[1,k]$ are degenerate. Since
  $$ \sB_1+\sU\sD\sV\sA_{k1}
              = \sU(\sI_r+\sD\cdot\sV\sA_{k1}\sB_1^{-1}\sU)\sU^{-1}\sB_1, $$
this matrix is non-degenerate. Therefore, it remains to show that there is a
diagonal matrix $\sD$ such that
  $$ \sB_2+\sU\sD\sV\sA_{k2}\longc \sB_k+\sU\sD\sV\sA_{kk} $$
are non-degenerate.

To this end we consider the polynomial in $r$ variables over the field $\F$,
defined by
  $$ P(t_1\longc t_r) := \prod_{j=2}^k \det(\sB_j+\sU\sD\sV\sA_{kj});
                                         \quad  \sD=\diag(t_1\longc t_r). $$
This is a non-zero polynomial, as each factor is distinct from $0$:
indeed, the coefficient of $t_1\dotsb t_r$ in the last factor is
$\det(\sU\sV\sA_{kk})\ne 0$, and for each $j\in[2,k-1]$ the constant term
of the $j$\,th factor is $\det(\sB_j)\ne 0$. Furthermore, each factor is
linear in every variable $t_i$; hence the degree of $P$ in every variable
is $k-1$. It is well-known, however, that a non-zero polynomial in $r$
variables cannot vanish on a cartesian product of $r$ sets, provided that
the cardinality of each set exceeds the degree of the polynomial in the
corresponding variable; see, for instance, \cite[Lemma~2.1]{b:at}. This
implies the existence of $t_1\longc t_r\in\F$ with $P(t_1\longc t_r)\ne
0$, and hence the existence of $\sD$ with the required property.
\end{proof}

\bigskip

\end{document}